\documentclass[1pt]{article}
\usepackage{amsmath, amssymb, amsthm, amsfonts, cases}
\usepackage{mathrsfs}
\usepackage{url}
\usepackage{authblk}
\usepackage[usenames]{color}
\usepackage{geometry}
\allowdisplaybreaks

\theoremstyle{plain}
\newtheorem{theorem}{Theorem}[section]

\newtheorem{problem}[theorem]{Problem}
\newtheorem{lemma}[theorem]{Lemma}

\theoremstyle{definition}
\newtheorem{definition}{Definition}[section]
\newtheorem{assumption}{Assumption}[section]

\newcommand{\epss}{\varepsilon}

\newcommand{\la}{\langle}
\newcommand{\ra}{\rangle}
\renewcommand{\theequation}{\thesection.\arabic{equation}}
\makeatletter\@addtoreset{equation}{section} \makeatother
\begin{document}

\title{  Optimal Control with State Constraints for Stochastic  Evolution Equation with Jumps
in Hilbert Space
 \thanks{This work was supported by the Natural Science Foundation of Zhejiang Province
for Distinguished Young Scholar  (No.LR15A010001),  and the National Natural
Science Foundation of China (No.11471079, 11301177) }}

\date{}

   \author{ Qingxin Meng\thanks{Corresponding author.   E-mail: mqx@zjhu.edu.cn}  \hspace{1cm}
   Qiuhong Shi
 \hspace{1cm}  Maoning Tang
\hspace{1cm}
\\\small{Department of Mathematics, Huzhou University, Zhejiang 313000, China}}

\maketitle
\begin{abstract}

This paper studies a stochastic
optimal control problem
with state constraint, where the state equation is described
by a controlled stochastic  evolution equation with jumps
in Hilbert Space
and the control domain is assumed to be convex. By means of Ekland variational
principle,  combining  the convex variation method
and the duality technique, necessary conditions
for optimality are derived in the form of stochastic maximum principles.
\end{abstract}


\maketitle

\section{Introduction}

In this paper, we study the optimal control for the following stochastic
evolution equation with jumps
\begin{eqnarray} \label{eq:1.1}
  \left\{
  \begin{aligned}
   d X (t) = & \ [ A (t) X (t) + b ( t, X (t), u(t)) ] d t
+ [B(t)X(t)+g( t, X (t), u(t)) ]d W(t)
 \\&\quad +\int_E \sigma (t, e,X(t-),u(t))\tilde \mu(de,dt),  \\
X (0) = & \  x , \quad t \in [ 0, T ]{\color{blue},}
  \end{aligned}
  \right.
\end{eqnarray}
with the cost functional
\begin{equation}\label{eq:1.2}
J(u(\cdot))= {\mathbb E} \bigg [ \int_0^T l ( t, X (t), u (t) ) d t
+ \Phi ( X (T) ) \bigg ],
\end{equation}
and  state constraint
\begin{equation} \label{eq:3.3}
{\mathbb E} [ \phi(X(T)) ] =0 ,
\end{equation}
in the framework of a Gelfand triple $V \subset H= H^*\subset V^{*},$ where $ H$ and  $V$ are
two given Hilbert spaces. Here on  a given  filtrated probability space  $(\Omega, \mathscr{F},\{
{\mathscr F}_t\}_{0\leq t\leq T}, P),$ $W$  is  a  one-dimensional  Brownian motion  and  $\tilde \mu$
is a Poisson random martingale measure on a fixed nonempty Borel measurable
subset ${E}$ of $\mathbb R^1,$
$A:[0,T]\times \Omega \longrightarrow {\mathscr L} (V, V^*)$, $B
  : [0,T]\times \Omega \longrightarrow {\mathscr L} (V, H ),$ $b:[0,T]\times  \Omega\times H
   \times  U_{ad}\longrightarrow H$, $g:[0,T]\times\Omega\times H
   \times  U_{ad} \longrightarrow H$ and $\sigma:[0,T]\times  \Omega \times
   E\times H  \times  U_{ad}\longrightarrow H$  are given random mappings, where the control variable $u$ takes value in a nonempty convex subset $U_{ad}$ of a real Hilbert space $U$. Here we denote by  $\mathscr{L}(V,V^*)$ the space of bounded
linear transformations of V into $V^*$, by ${\mathscr L} (V, H)$   the space of bounded
linear transformations of  $H$ into $V.$
    An adapted solution of
   \eqref{eq:1.1} is  a $V$-valued, $\{{\mathscr F}_t\}_{0\leq t\leq T}$-adapted process $X(\cdot)$  which satisfies \eqref{eq:1.1} under some appropriate sense.
   The optimal control problem is to find an admissible control to minimize the cost functional  \eqref{eq:1.2}
 over  the set of admissible controls.

One of the basic method to solve  stochastic optimal control
problems is the stochastic maximum principle whose objective is to
establish necessary (as well as sufficient) optimality conditions of
controls. For  optimal control problems of infinite dimensional
stochastic systems, many works are concerned with the
stochastic systems and the corresponding stochastic maximum
principles, see e.g.( \cite{HuPe901, Ben830, Zhou93,  Alhu10,
Alhu101, Alhu111, Alhu112, LuZh12,Guat11, DuMeng}.

 In contrast, there have not been
  a
 number of   results on  the optimal control for stochastic partial
differential equations driven by jump processes. In 2005,  {\O}ksendal, Proske,  Zhang \cite{Ok} studied the optimal control problem  of quasilinear semielliptic SPDEs driven by Poisson random measure   and gave sufficient maximum principle results, not necessary ones.  In 2017,  Tang and Meng \cite{Tangmeng2016}  studied the optimal control problem
for a controlled  stochastic evolution equation \eqref{eq:1.1} with the
cost functional \eqref{eq:1.2}, where  the  control domain is assumed to be  convex.  \cite{Tangmeng2016} adopt the  convex variation method and the first adjoint duality analysis to show a necessary maximum principle.  And  Under the convexity assumption of the Hamiltonian and the terminal cost, a sufficient maximum principle for this optimal problem which is the so-called verification theorem is  obtained

     The purpose of this paper is to  establish the
 maximum principle
 for the optimal control problem where the state process is driven by a  controlled stochastic evolution equation \eqref{eq:1.1} with the cost
 functional \eqref{eq:1.2} and the
 state constraint \eqref{eq:3.3} by Ekland variational
principle,  combining  the convex variation method
and the duality technique.

The paper is organized as follows. In section 2 we formulate the problem and give
various assumptions used throughout the paper. In section 3, we present a penalized optimal control problem. Section 4 is  devoted to derive necessary  optimality
conditions in the form of stochastic maximum principles in a unified
way. Some basic results on the SEE and the BSEE with jump are given in the Appendix which will been
used in this paper.
\section{Problem formulation}

In this section, we introduce basic notation and standing assumptions, and state an optimal control
problem with state constraint under a  stochastic evolution equation with jumps
in Hilbert space, which was considered by Tang and Meng\cite{Tangmeng2016}.

Let $(\Omega, \mathscr{F}, \mathbb P)$ be a complete probability space
equipped with a  one-dimensional standard Brownian motion $\{W(t),
0\leq t\leq T\}$  and  a stationary Poisson point process
$\{\eta_t\}_{t\geq 0}$ defined  on a fixed nonempty Borel measurable
subset ${E}$ of $\mathbb R^1$.
 Denote by $\mathbb E[\cdot]$ the expectation
under the probability $\mathbb P.$
 We denote by $\mu(de,dt)$
 the counting measure induced by $\{\eta_t\}_{t\geq 0}$ and
  by  $\nu(de)$ the corresponding
 characteristic measure.  Then the compensatory
 random martingale measure is denoted by
  $\tilde{\mu}(de, dt):={\mu}(de,
dt)-\nu(de)dt$ which is assumed to be independent of the Brownian
motion  $\{W(t),
0\leq t\leq T\}$.
  Furthermore, we assume that
$\nu({E})<\infty$. Let $\{{\mathscr F}_t\}_{0\leq t\leq T}$ be the
P-augmentation of the natural filtration generated by
$\{{W_t}\}_{t\geq 0}$ and $\{\eta_t\}_{t\geq 0}$.
 By  $\mathscr{P}$  we denote the
predictable $\sigma$ field on $\Omega\times [0, T]$ and
by $\mathscr B(\Lambda)$
  the Borel $\sigma$-algebra of any topological space $\Lambda.$
  Let $X$ be  a  separable Hilbert space with norm $\|\cdot\|_X$.
  Denote by  $M^{\nu,2}( E; X)$ the set of all $X$-valued measurable
  functions $r=\{r(e), e\in E\}$ defined on the measure
  space $(E, \mathscr B(E); v)$
  such that
$\|r\|_{M^{\nu,2}( E; X)}\triangleq
\sqrt{{\int_E\|r(e)\|_X^2v(de)}}<~\infty,$ by
${M}_\mathscr{F}^{\nu,2}{([0,T]\times  E; X)}$ the  set of all
$\mathscr{P}\times {\mathscr B}(E)$-measurable $X$-valued processes
$r=\{r(t,\omega,e),\
(t,\omega,e)\in[0,T]\times\Omega\times E\}$ such that
$\|r\|_{{M}_\mathscr{F}^{\nu,2}{([0,T]\times  E; X)}}\triangleq
\sqrt{{\mathbb E\bigg[\displaystyle\int_0^T\displaystyle\int_E\displaystyle\|r(t,e)\|_X^2
\nu(de)dt\bigg]}}<~\infty,$
 by
$M_{\mathscr{F}}^2(0,T;X)$ the set of all ${\mathscr{F}}_t$-adapted
$X$-valued  processes $f=\{f(t,\omega),\
(t,\omega)\in[0,T]\times\Omega\}$ such that
$\|f\|_{M_{\mathscr{F}}^2(0,T;X)}
\triangleq\sqrt{\mathbb E\bigg[\displaystyle\int_0^T\|f(t)\|_X^2dt\bigg]}<\infty,$
by  $S_{\mathscr{F}}^2(0,T;X)$  the set of all
${\mathscr{F}}_t$-adapted  $X$-valued c\`{a}dl\`{a}g processes
$f=\{f(t,\omega),\ (t,\omega)\in[0,T]\times\Omega\}$ such that
$\|f\|_{S_{\mathscr{F}}^2(0,T;X)}\triangleq\sqrt{
\mathbb E\bigg[\displaystyle\sup_{0
\leq t \leq T}\|f(t)\|_X^2}\bigg]<+\infty,$  by
$L^2(\Omega,{\mathscr{F}},\mathbb P;X)$ the set of all $X$-valued random
variables $\xi$ on $(\Omega,{\mathscr{F}},
\mathbb P)$ such that
$\|\xi\|_{L^2(\Omega,{\mathscr{F}},
\mathbb P};X)\triangleq
\sqrt{\mathbb E[\|\xi\|_X^2]}<\infty.$
Throughout this paper, we let  $C$ and $K$  be two generic positive constants, which may be different from line to line.

In what follows, we set up a Gelfand triple $(V, H, V^*)$, based on which the state process and the adjoint process is defined.
Indeed, the state process is governed by a SEE with jumps, while the adjoint process is governed by a BSEE with jumps. We provide the existence, uniqueness
and continuous dependence theorems for SEEs with jumps  and BSEEs with jumps in the appendix.

Let $V$ and $H$ be two separable (real) Hilbert spaces such that $V$
is densely embedded in $H$. We identify $H$ with its dual space by the Riesz mapping. Then we can take $H$ as a pivot space and get a Gelfand triple $(V, H,
V^*)$ such that $V \subset H = H^*\subset V^{*}$. Let $(\cdot,\cdot)_H$ denote the inner product in $H$, and $\la\cdot,\cdot\ra$ denote the duality product between
$V$ and $V^{*}$. Moreover, we write $\mathscr{L}(V,V^*)$ for the space of bounded linear transformations of V into $V^*$.

The state process is governed by the following controlled SEE with jumps in the Gelfand triple $(V, H, V^*)$:
\begin{eqnarray} \label{eq:4.1}
  \left\{
  \begin{aligned}
   d X (t) = & \ [ A (t) X (t) + b ( t, X (t), u(t)) ] d t
+ [B(t)X(t)+g( t, X (t), u(t)) ]d W(t)
 \\&\quad +\int_E \sigma (t, e,X(t-),u(t))\tilde \mu(de,dt),  \\
X (0) = & \  x , \quad t \in [ 0, T ]{\color{blue},}
  \end{aligned}
  \right.
\end{eqnarray}
where the space of controls $U_{ad}$ is given by a nonempty closed convex subset of a separable real Hilbert space $U$.

\begin{definition}
A stochastic process  $u(\cdot)$ is an admissible control, if $u(t)\in U_{ad}$ for almost $t\in [0, T]$ and $ u(\cdot)\in
M_{\mathbb F}^2(0, T;U)$. The set of all admissible controls is denoted by ${\cal A}$.
\end{definition}

The cost functional is given by
\begin{equation}\label{eq:2.2}
J(u(\cdot))= {\mathbb E} \bigg [ \int_0^T l ( t, x (t), u (t) ) d t
+ \Phi ( x (T) ) \bigg ].
\end{equation}
We assume that the control system (\ref{eq:4.1})-(\ref{eq:2.2}) is subject to the following state constraint
\begin{equation} \label{eq:2.555}
{\mathbb E} [ \phi(X(T)) ] =0 .
\end{equation}
Here the coefficients $(A,B, b,g,\sigma, l,\Phi, \phi)$ of the control system (\ref{eq:4.1})-(\ref{eq:3.3})
\begin{assumption}\label{ass:2.5}
\begin{enumerate}
\item[]
\item[(i)]

The operator processes $A:[0,T]\times \Omega \longrightarrow {\mathscr L} (V, V^*)$ and $B
  : [0,T]\times \Omega \longrightarrow {\mathscr L} (V, H)$
  are weakly predictable; i.e.,
  $ \langle A(\cdot)x, y \rangle$ and $(B(\cdot)x, y)_H$
  are both predictable process for every $x, y\in V, $
  and satisfy the coercive condition, i.e.,  there exist
  some constants  $ C, \alpha>0$ and $\lambda$ such that for any $x\in V$ and  each $(t,\omega)\in [0,T]\times \Omega,$
    \begin{eqnarray}
    \begin{split}
     - \langle A(t)x, x \rangle +\lambda ||x||_H^2\geq \alpha
      ||x||_V^2+||Bx||_H^2,
    \end{split}
  \end{eqnarray}
  and  \begin{eqnarray} \label{eq:3.3}
\sup_{( t, \omega ) \in [0, T] \times \Omega} \| A ( t,\omega ) \|_{{\mathscr L} ( V, V^* )}
 +\sup_{( t, \omega ) \in [0, T] \times \Omega} \| B ( t,\omega ) \|_{{\mathscr L} ( V, H )} \leq C \ .
\end{eqnarray}

\item[(ii)]$b, g: [ 0, T ] \times \Omega \times H \times {\mathscr U} \rightarrow H$  are $\mathscr P\times
   \mathscr B(H)\times \mathscr B(\mathscr U)/\mathscr B(H) $ measurable
   mappings and $\sigma:[0,T]\times  \Omega \times
   E\times H \times \mathscr U \longrightarrow H$ is a $\mathscr P\times \mathscr B(E)\times
   \mathscr B(H)\times \mathscr B(U)/\mathscr B(H)$-measurable mapping  such that $b ( \cdot, 0, 0 ), g ( \cdot, 0, 0 ) \in {
M}^2_{\mathscr F} ( 0, T; H ), \sigma(\cdot, \cdot, 0,0)\in {M}_\mathscr{F}^{\nu,2}{([0,T]\times  E; H)}.$ Moreover, for almost all $( t, \omega, e ) \in [ 0, T ] \times \Omega \times E$,  $b$, $g$ and $\sigma$ are  G\^ateaux differentiable in $(x,u)$ with  continuous bounded G\^ateaux  derivatives
$b_x, g_x,\sigma_x,  b_u, g_u$ and  $\sigma_u$;
\item[(iii)]
$l:[ 0, T ] \times \Omega \times H \times {\mathscr U} \rightarrow \mathbb R $ is a
 ${\mathscr P} \otimes {\mathscr B} (H) \otimes {\mathscr B} ({\mathscr U})/ {\mathscr B} ({\mathbb R})$-measurable mapping
  and  $\Phi,\phi: \Omega \times H \rightarrow {\mathbb R}$
is  a ${\mathscr F}_T\otimes {\mathscr B} (H) / {\mathscr B} ({\mathbb R})$-measurable
mapping.
For almost all $( t, \omega ) \in [ 0, T ] \times \Omega$, $l$ is continuous G\^ateaux
differentiable in $(x,u)$
with continuous  G\^ateaux derivatives $l_x$ and $l_u$, and $\Phi$ and $\phi$
 are G\^ateaux differentiable in $x$
with  continuous
 G\^ateaux derivative $\Phi_x$ and
 $\phi_x$.
Moreover, for almost all $( t, \omega ) \in [ 0, T ] \times \Omega$,   there exists a   constant $C > 0$ such that  for all $( x, u ) \in H  \times {\mathscr U}$
\begin{eqnarray*}
| l ( t, x, u  ) |
\leq  C ( 1 + \| x \|^2_H + + \| u \|_U^2 ) ,
\end{eqnarray*}
\begin{eqnarray*}
&& \| l_x ( t, x,u) \|_H +
+ \| l_u ( t, x, u ) \|_U \leq C ( 1 + \| x \|_H  + \| u \|_U  ) ,
\end{eqnarray*}
and
\begin{eqnarray*}
& | \Phi (x) | \leq C ( 1 + \| x \|^2_H) , | \phi (x) | \leq C ( 1 + \| x \|^2_H)  \\
& \| \Phi_x (x) \|_H \leq C ( 1 + \| x \|_H),\| \phi_x (x) \|_H \leq C ( 1 + \| x \|_H).
\end{eqnarray*}
\end{enumerate}
\end{assumption}

Under Assumption \ref{ass:2.5}, it can be shown from Lemma \ref{thm:3.1} that for any $u(\cdot)\in {\cal A},$
the state equation \eqref{eq:4.1} admits a unique solution $X(\cdot)\in M_{\mathscr{F}}^2(0,T;V)\bigcap S_{\mathscr{F}}^2(0,T;H)$. We also denote this solution as
$X^u(\cdot)$ whenever we want to emphasis its dependence on the control $u(\cdot)$. Then we call
$X(\cdot)$ the state process corresponding to the control process $u(\cdot)$ and $(u(\cdot); X(\cdot))$
the admissible pair. Furthermore, from  Assumption \ref{ass:2.5} and the a priori estimate \eqref{eq:3.4}, we can easily validate that
\begin{eqnarray*}
|J(u(\cdot))|<\infty.
\end{eqnarray*}

Now we state formally the optimal control problem
\begin{problem}\label{pro:2.1}
Find an admissible control $\bar{u}(\cdot)$ such that
\begin{eqnarray*}\label{eq:b7}
J(\bar{u}(\cdot))=\inf_{u(\cdot)\in {\cal A}}J(u(\cdot)),
\end{eqnarray*}
subject to \eqref{eq:4.1} and \eqref{eq:2.555}, where the cost functional is given by \eqref{eq:2.2}.
\end{problem}

Any $\bar{u}(\cdot)\in {\cal A}$ satisfying the above is called an optimal control process of Problem \ref{pro:2.1};
the corresponding state process $\bar{X}(\cdot)$ is called an optimal state process; correspondingly,
$(\bar{u}(\cdot); \bar{X}(\cdot))$ is called an optimal pair of Problem \ref{pro:2.1}.

\section{Penalized optimal control problem}

In this section, we relate the original constrained control problem with one without state constraint.

The results relies on the following Ekeland's principle.

\begin{lemma}[Ekeland's principle, \cite{ekeland1974variational}]
Let $(S,d)$ be a complete metric space and $\rho (\cdot ):S
\rightarrow {\mathbb R}$ be lower-semicontinuous and bounded from
below. For $\varepsilon \geq 0$, suppose $u^{\varepsilon }\in S$
satisfies
\begin{equation*}
\rho ( u^\varepsilon ) \leq \inf_{u\in S} \rho (u) + \varepsilon .
\end{equation*}
Then for any $\lambda >0$, there exists $u^\lambda \in S$ such that
\begin{eqnarray*}
\rho ( u^\lambda ) \leq \rho ( u^\varepsilon ) , \quad d (
u^\lambda, u^\varepsilon  ) \leq \lambda ,
\end{eqnarray*}
and
\begin{eqnarray*}
\rho ( u^\lambda ) \leq \rho (u) + \frac{\varepsilon}{\lambda} d
(u^\lambda, u) , {\mbox { for all }} u \in S.
\end{eqnarray*}
\end{lemma}

 Define
a metric $d$ on  the admissible
 controls set $\cal A$ as
\begin{eqnarray}\label{eq:3.13}
d (u_1 (\cdot), u_2 (\cdot)) \triangleq \bigg \{ {\mathbb E} \bigg [ \int_0^T||u_1 (t) - u_2 (t)||^2_U dt\bigg] \bigg \}^{\frac{1}{2}} ,
\quad \forall u_1(\cdot), u_2 (\cdot) \in \cal A.
\end{eqnarray}

We can assume   that
$\cal A$ is a bounded closed convex set in the sense of \eqref{eq:3.13},
the unbounded case can be reduced to the
bounded case.

Under this assumption of boundedness and closedness
of $\cal A$,
we have the following basic
lemma which will be used in the sequence.

\begin{lemma}\label{lem:4.2}
$(\Lambda, d)$ is a complete metric space.
\end{lemma}

\begin{proof}
Since the control space $U$ is a Hilbert space $M_{\mathbb{F}}^2(0,T;U)$ is also a Hilbert space under \eqref{eq:3.13}. Therefore, $\cal A$ is complete under the distance defined by \eqref{eq:3.13}.
since $\cal A$ is a  closed subset of $M_{\mathbb{F}}^2(0,T;U).$
The proof is complete.
\end{proof}

The next lemma shows that a mapping from the control process in $\cal A$ to the state process in ${\cal M}^2_{\mathbb F} (0, T)$, to be defined below, is bounded and continuous.
To simplify our notation, we write
\begin{eqnarray}
{\cal M}^2_{\mathscr F} (0, T) \triangleq S_ {\mathscr{F}}^2(0,T;H) \cap M_{\mathscr{F}}^2(0,T;V)
\end{eqnarray}
and
\begin{eqnarray}
|| X(\cdot)||_{{\cal M}^2_{\mathscr F} (0, T)} \triangleq \sqrt{||X(\cdot)||^2_{S^2_{\mathscr F}(0,T;H)}
+ ||X(\cdot)||^2_{M^2_{\mathscr F}(0,T;V)} }.
\end{eqnarray}
The next lemma shows that a mapping from the control process in ${\cal A}$ to the state process in $ M_{\mathscr{F}}^2(0,T;V)$ is bounded and continuous.

\begin{lemma}\label{lem:4.7}
Let Assumption \ref{ass:2.5} be satisfied. Then the mapping ${\cal I}: ({\cal A}, d) \rightarrow ({\cal M}_{\mathscr{F}}^2(0,T), || \cdot ||_{{\cal M}_{\mathscr{F}}^2(0,T)})$
defined by
\begin{eqnarray*}
{\cal I} (u (\cdot)) = X^u(\cdot)
\end{eqnarray*}
is bounded and continuous.
\end{lemma}

\begin{proof}
By the a priori estimate of SEE (Lemma \ref{thm:3.2}), it can be shown that for any $u(\cdot)\in \Lambda$,
\begin{eqnarray}\label{eq:5.7}
|| X^u(\cdot) ||^2_{{\cal M}_{\mathscr{F}}^2(0,T)}
&\leq& K \bigg \{ {\mathbb E} [ ||x||^2_H ] + {\mathbb E} \bigg [ \int_0^T|| u(t)|| _{U}^{2}dt \bigg ] + 1 \bigg \} \nonumber \\
&\leq& K.
\end{eqnarray}
Here $K$ is a positive constant independent of $u(\cdot)$ and may change from line to line.

On the other hand, let $\{v_n(\cdot)\}_{n \geq 1}$ be a sequence in $\cal A$ such that it  converges an admissible  $v(\cdot)\in \cal A $ under the metric $d$. Suppose that $X_n(\cdot)$, for each $n = 1, 2, \cdots$,
and $X(\cdot)$ are the state processes corresponding to $v_n(\cdot)$ and $v(\cdot)$, respectively. By making use of the a priori estimate of SEE (Lemma \ref{thm:3.2}), we can deduce that
\begin{eqnarray}\label{eq5.9}
&&|| X^{v_n}(\cdot) - X^v(\cdot)||^2_{{\cal M}_{\mathscr{F}}^2(0,T)}
\\&\leq& K \bigg\{{\mathbb E} \bigg [\int_{0}^{T}||b(t,X^{v}(t), v_n(t))-b(t,X^{ v}(t), v(t)  )||^2_H dt
 \bigg]+
 {\mathbb E} \bigg [\int_{0}^{T}||g(t,X^{v}(t), v_n(t))-g(t,X^{ v}(t), v(t)  )||^2_H dt
 \bigg]\nonumber
 \\&&\quad\quad+
 {\mathbb E} \bigg [\int_{0}^{T}||\sigma(t,X^{v}(t), v_n(t))-\sigma (t,X^{ v}(t), v(t)  )||^2_{M^{\nu,2}( E; H)} dt
 \bigg]\bigg\}
  \nonumber \\
&\leq& K {\mathbb E} \bigg [\int_{0}^{T}||v_n(t)- v(t)||^2_U dt\bigg] \nonumber \\
&=& K d^2(v_n(\cdot), v(\cdot)).
\end{eqnarray}
Sending $n\rightarrow\infty$ in (\ref{eq5.9}) yields
\begin{eqnarray}\label{eq:5.11}
|| X^{v_n}(\cdot) - X^v(\cdot)||^2_{{\cal M}_{\mathscr{F}}^2(0,T)} \rightarrow 0 .
\end{eqnarray}
This validates the continuity of ${\cal I}$.
\end{proof}

\begin{lemma} \label{lem:4.4}
Let Assumption \ref{ass:2.5} be satisfied. Then the cost functional $J(u(\cdot))$ is bounded and continuous on $\cal A$ under the metric \eqref{eq:3.13}.
\end{lemma}

\begin{proof}
For any $u(\cdot)\in {\cal A}$, under Assumption \ref{ass:2.5} and from Lemma \ref{lem:4.7} we have
\begin{eqnarray}
|J(u(\cdot))| &\leq& {\mathbb E} \bigg [ \int_0^T |l(t,X^u(t),u(t))|dt + |\Phi(X^u(T))| \bigg ] \nonumber \\
&\leq& K \bigg [ 1 + ||X^u(\cdot)||^2_{M_{\mathscr{F}}^2(0,T;V)}
+ ||u(\cdot)||_{M_{\mathscr{F}}^2(0,T;U)}^2 + ||X(T)||^2_{L^2(\Omega,{\mathscr{F}_T},
\mathbb P;H)} \bigg] \nonumber \\
&\leq& K \bigg [ 1 + ||X^u(\cdot)||^2_{{\cal M}^2_{\mathscr F}(0,T)} + ||{ u}(\cdot)||_{M^2_{\mathscr F}(0,T;U)}^2 \bigg] \nonumber \\
&\leq& K .
\end{eqnarray}
Here $K$ is a positive constant independent of $u(\cdot)$ and  may change from line to line. This implies the cost
functional  $J(u(\cdot))$ is bounded on ${\cal A}$.

To show the continuity of the cost functional, as in the proof of Lemma \ref{lem:4.7} we pick up the sequence $\{v_n(\cdot)\}_{n \geq 1}$
and its converging point $v(\cdot)$ in $\cal A$ as well as the corresponding state processes $X_n(\cdot)$ and $X(\cdot)$.
Thus using Lemma \ref{lem:4.7} and the Lebesgue dominated convergence theorem, we obtain
\begin{eqnarray}\label{eq:3.4}
J(v^n(\cdot))\rightarrow J(v(\cdot)) , \quad \mbox {as} \ n \rightarrow \infty .
\end{eqnarray}
The completes the proof.
\end{proof}

Define a penalized cost functional associated with Problem \eqref{pro:2.1} as
\begin{equation}\label{eq:3.2000}
J^\varepsilon(v(\cdot)) \triangleq \bigg\{ \big [ J(v(\cdot)) -J(\bar u(\cdot)) + \varepsilon \big ]^2+ \big | {\mathbb E} [ \phi( X^v(T))] \big |^2 \bigg\}^{\frac{1}{2}} ,
\quad \forall \varepsilon>0 .
\end{equation}
It is worthwhile to point out that we will study this functional over $\cal A$.

\begin{lemma}\label{lem:4.5}
$J^\varepsilon(v(\cdot))$ is bounded and continuous on ${\cal A}$ under the metric \eqref{eq:3.13}.
\end{lemma}

\begin{proof}
The proof can be obtained by Lemma \ref{lem:4.4} and Lemma \ref{lem:4.7} immediately.
\end{proof}

Now we introduce an auxiliary optimal control problem without state constraint:

\begin{problem}[$(SC)^\varepsilon$]
Find an admissible control such that
\begin{eqnarray}
\inf_{v(\cdot)\in {\cal A}} J^\varepsilon(v(\cdot)),
\end{eqnarray}
where the state process is given by (\ref{eq:4.1}) and the cost functional $J^\varepsilon(v(\cdot))$ is given by (\ref{eq:3.4}).
\end{problem}

From the definition of the penalized cost functional (\ref{eq:3.2000}), we see that
\begin{eqnarray}
J^\varepsilon(\bar u(\cdot)) = \varepsilon
 \leq  \inf_{v(\cdot)\in {\cal A}} J^\varepsilon(v(\cdot))+\varepsilon.
\end{eqnarray}
An application of Ekeland's variational principle shows that there is a $u^\varepsilon(\cdot)\in {\cal A}$ such that
\begin{eqnarray}\label{eq:5.1}
\left\{
\begin{aligned}
& J^\varepsilon(u^\varepsilon(\cdot))\leq J^\varepsilon(\bar u(\cdot))=\varepsilon  , \\
& d(u^\varepsilon(\cdot),\bar u(\cdot))\leq \varepsilon^{\frac{1}{2}}, \\
& J^\varepsilon(v(\cdot))- J^\varepsilon(u^\varepsilon(\cdot))\geq -\varepsilon^{\frac{1}{2}}d(u^\varepsilon(\cdot),v(\cdot)), \quad \forall v(\cdot)\in {\cal A}.
\end{aligned}
\right.
\end{eqnarray}
Define a convex perturbed control of ${u}^{\varepsilon }\left( \cdot \right)$ as
\begin{eqnarray}\label{eq:5.2}
u^{\varepsilon ,\rho}\left( \cdot \right) \triangleq {u}^{\varepsilon }(\cdot )+\rho(%
u\left( \cdot \right)-u^{\varepsilon }(\cdot ) ),
\end{eqnarray}
where $u\left( \cdot \right)$ is an arbitrary admissible control in ${\cal A}$ and $0\leq \rho\leq 1$. It is easy to verify that
$u^{\varepsilon,\rho }\left( \cdot \right)$ is also in ${\cal A}$. Suppose that $X^{\epss,\rho}(\cdot)$ and
$X^{\epss}(\cdot)$ are the state processes corresponding to $u^{\epss, \rho}(\cdot)$ an $ u^\epss(\cdot)$, respectively.
By \eqref{eq:5.1} and the fact
\begin{eqnarray}\label{eq:5.3}
d\left( u^{\varepsilon ,\rho}\left( \cdot \right) ,{u}^{\varepsilon
}\left( \cdot \right) \right) \leq C\rho,
\end{eqnarray}
we have
\begin{eqnarray}\label{eq:3.9}
J^\varepsilon(u^{\varepsilon,\rho}(\cdot))-J^\varepsilon(u^\varepsilon(\cdot))
\geq {-\varepsilon ^{\frac{1}{2}}d\left( u^{\varepsilon ,\rho}\left(t\right),{u}^{\varepsilon}\left( t\right) \right) }
\geq -\varepsilon^{\frac{1}{2}}C\rho.
\end{eqnarray}

On the other hand,  from the definition of $J^\varepsilon(\bar u(\cdot))$, we have
\begin{eqnarray}\label{eq:3.17}
J^\varepsilon(u^{\varepsilon,\rho}(\cdot))-J^\varepsilon(u^\varepsilon(\cdot))
&=& \frac{[J^\varepsilon(u^{\varepsilon,\rho}(\cdot))]^2-[J^\varepsilon(u^\varepsilon(\cdot))]^2}
{J^\varepsilon(u^{\varepsilon,\rho}(\cdot))+J^\varepsilon(u^\varepsilon(\cdot))} \nonumber \\
&=& \frac {J(u^{\varepsilon,\rho}(\cdot))+J(u^\varepsilon(\cdot))-2J(\bar u(\cdot))+2\varepsilon}
{J^\varepsilon(u^{\varepsilon,\rho}(\cdot))+J^\varepsilon(u^\varepsilon(\cdot))}\times[J(u^{\varepsilon,\rho}(\cdot))-J(u^\varepsilon(\cdot))] \nonumber \\
&& +\frac {{\mathbb E}[\phi(X^{\varepsilon,\rho}(T))]+{\mathbb E}[\phi(X^\varepsilon(T))]}{J^\varepsilon(u^{\varepsilon,\rho}(\cdot))+J^\varepsilon(u^\varepsilon(\cdot))}
\times \big \{ {\mathbb E}[\phi(X^{\varepsilon,\rho}(T))]-{\mathbb E}[\phi(X^\varepsilon(T))] \big \} \nonumber \\
&=& {\lambda}^{\varepsilon,\rho}[{J(u^{\varepsilon,\rho}(\cdot))-J(u^\varepsilon(\cdot))}]
+\mu^{\varepsilon,\rho} \big \{ {\mathbb E}[\phi(X^{\varepsilon,\rho}(T))]-{\mathbb E}[\phi(X^\varepsilon(T))] \big \} ,
\end{eqnarray}
where
\begin{eqnarray}
{\lambda}^{\varepsilon,\rho}\triangleq\frac {J(u^{\varepsilon,\rho}(\cdot))+J(u^\varepsilon(\cdot))-2J(\bar
u(\cdot))+2\varepsilon}{J^\varepsilon(u^{\varepsilon,\rho}(\cdot))+J^\varepsilon(u^\varepsilon(\cdot))}
\end{eqnarray}
and
\begin{eqnarray}
\mu^{\varepsilon,\rho}\triangleq\frac{{\mathbb E}[\phi(X^{\varepsilon,\rho}(T))]+{\mathbb E}[\phi(X^\varepsilon(T))]}
{J^\varepsilon(u^{\varepsilon,\rho}(\cdot))+J^\varepsilon(u^\varepsilon(\cdot))} .
\end{eqnarray}
From \eqref{eq:5.3}, we have
\begin{eqnarray} \label{eq:5.8}
\lim_{\rho\rightarrow 0} d\left( u^{\varepsilon ,\rho}\left( \cdot \right) ,{u}^{\varepsilon}\left( \cdot \right) \right) =0
\end{eqnarray}
Then it follows from Lemma \ref{lem:4.4} and Lemma \ref{lem:4.5} that
\begin{eqnarray} \label{eq:5.9}
\lim_{\rho\rightarrow 0} ||X^{\epss,\rho}(\cdot)-X^{ \epss}(\cdot)||_{{\cal M}^2_{\mathscr F}(0,T)}^2=0
\end{eqnarray}
and
\begin{eqnarray}
\lim_{\rho\rightarrow 0} J^{\epss}(u^{\epss,\rho}(\cdot))= J^{\epss}(u^{\epss}(\cdot)) .
\end{eqnarray}
Consequently,
\begin{eqnarray}\label{eq:5.10}
\lim_{\rho\rightarrow 0} { \lambda}^{\varepsilon,\rho}= {\lambda}^{\varepsilon}, \quad \lim_{\rho \rightarrow 0} \mu^{\varepsilon,\rho}= \mu^{\varepsilon} ,
\end{eqnarray}
where
\begin{eqnarray}
{\lambda}^{\varepsilon}\triangleq\frac {J(u^{\varepsilon}(\cdot))-J(\bar u(\cdot))+\varepsilon}{J^\varepsilon(u^\varepsilon(\cdot))}
\end{eqnarray}
and
\begin{eqnarray}
\mu^{\varepsilon}\triangleq\frac{{\mathbb E}[\phi(X^{\varepsilon}(0))]}{J^\varepsilon(u^{\varepsilon}(\cdot))} .
\end{eqnarray}
Note that
\begin{eqnarray}
|{\lambda}^\epss|^2+|\mu^\epss|^2=1.
\end{eqnarray}

Therefore, there exists a subsequence $\{({\lambda}^\epss, \mu^\epss)\}_{\epss > 0}$   ( still denoted also by $\{({\lambda}^\epss, \mu^\epss)\}_{\epss > 0}$, such that
\begin{eqnarray} \label{eq:5.13}
\lim_{\epss\rightarrow 0} {\lambda}^{\varepsilon}= {\lambda}, \quad \lim_{\epss \rightarrow0} \mu^{\varepsilon}= \mu ,
\end{eqnarray}
and
\begin{eqnarray} \label{eq:5.14}
|{\lambda}|^2+|\mu|^2=1.
\end{eqnarray}

\section{Stochastic Maximum Principle}

In this section, we first drive a variational formula for the penalized cost functional $J^\epss(u(\cdot))$.

To simplify our notation, we write partial derivatives of $b,g \sigma$ and $l$ as
\begin{eqnarray*}
& \varphi_x^{\epss,\rho}(t) \triangleq \varphi_x (t,{X}^{\varepsilon,\rho}(t),{u}^{\varepsilon,\rho}(t)), \\
& \varphi_x^{\epss}(t) \triangleq \varphi_a (t,{X}^{\varepsilon}(t),{u}^{\varepsilon}(t)), \\
& \bar\varphi_x(t) \triangleq \varphi_a (t,\bar{X}(t),\bar{u}(t)),
\end{eqnarray*}
where $\varphi = b, g,\sigma$ and $l$.

Define the Hamiltonian ${\cal H}: [ 0, T ] \times \Omega \times H  \times {\mathscr U} \times H\times H \times  M^{\nu,2}( E; H)\times \mathbb R
\rightarrow {\mathbb R}$ by
\begin{eqnarray}\label{eq:5.3}
{\cal H} ( t, x, u, p, q, r(\cdot),\lambda ) := \left ( b ( t, x, u ), p \right )_H
+\left( g ( t, x,  u), q \right)_H
+\int_{E}\left( \sigma ( t,e, x,  u), r(t,e) \right)_H\nu(de)
+ \lambda l ( t, x, u ) .
\end{eqnarray}
Using Hamiltonian ${\cal H}$,
the adjoint equation \eqref{eq:4.4}
can be written in the following form:

\begin{eqnarray}\label{eq:5.4}
\begin{split}
   \left\{\begin{array}{ll}
d\bar p(t)=&-\bigg[A^*(t)\bar p(t)+B(t)^*\bar q(t)+\bar {\cal H}_{x} (t)\bigg]dt+\bar q(t)dW(t)+\displaystyle
\int_{{E}}\bar r(t, e)\tilde{\mu}(de, dt),
~~~~0\leqslant t\leqslant T,
\\ \bar p(T)=&\Phi_x( \bar X(T)),
  \end{array}
 \right.
 \end{split}
  \end{eqnarray}
where we denote
\begin{eqnarray}\label{eq:5.6}
\bar {\cal H} (t) \triangleq {\cal H} ( t,
\bar x (t), \bar u (t), \bar p (t),
\bar q (t),
\bar r(t,\cdot) ).
\end{eqnarray}

Similarly, for notational simplify, we write partial derivatives of $H$ as
\begin{eqnarray*}
&{\cal H}_a^{\epss, \rho}(t) \triangleq
{\cal H}_a (t,{X}^{\epss, \rho}(t),{u}^{\epss, \rho}(t),{p}^{\epss, \rho}(t),q^{\epss, \rho}(t),r^{\epss, \rho}(t,\cdot),{\lambda}^{\epss, \rho}), \\
&{\cal H}_a^{\epss}(t) \triangleq H_a (t,{X}^{\epss}(t),{u}^{\epss}(t),
{p}^{\epss}(t),q^{\epss}(t),r^{\epss}(t,\cdot),
{\lambda}^{\epss}), \\
&{\bar{\cal H}}_a (t) \triangleq H_a (t,\bar{X}(t),\bar{u}(t),\bar{p}(t),{\bar q}(t), \bar r(t,\cdot),{\lambda}) .
\end{eqnarray*}
where $a = x$ or $u$.

For the admissible pair $({u}^{\varepsilon, \rho}(\cdot);{X}^{\varepsilon,\rho}(\cdot))$ and $({u}^{\varepsilon}(\cdot);
{X}^{\varepsilon}(\cdot))$ and the optimal pair $(\bar{u}(\cdot);\bar{X}(\cdot))$, the corresponding adjoint processes are denoted by
$\{ (p^{\epss,\rho}(t), q^{\epss,\rho}(t),
r^{\epss,\rho}(t,\cdot)), 0 \leq t \leq T \}$, $\{ (p^{\epss}(t), q^{\epss}(t),
r^{\epss,\rho}(t)), 0 \leq t \leq T\}$ and $\{ {\bar p}(t), \bar q(t),
 \bar r(t,\cdot)),0 \leq t \leq T \}$.
We now define the adjoint equations for $\{ (p^{\epss,\rho}(t), q^{\epss,\rho}(t),
r^{\epss,\rho}(t,\cdot)), 0 \leq t \leq T \}$, $\{ (p^{\epss}(t), q^{\epss}(t),
r^{\epss,\rho}(t)), 0 \leq t \leq T\}$ and $\{ {\bar p}(t), \bar q(t),
 \bar r(t,\cdot)),0 \leq t \leq T \}$ as
\begin{eqnarray}\label{eq:4.4}
\left\{
\begin{aligned}
d p^{\epss,\rho}(t)=&-\bigg[A^*(t) p^{\epss,\rho}(t)+B(t)^* q^{\epss,\rho}(t)+ {\cal H}_{x}^{\epss,\rho} (t)\bigg]dt+ q^{\epss,\rho}(t)dW(t)+\displaystyle
\int_{{E}}r^{\epss,\rho}(t, e)\tilde{\mu}(de, dt),
~~~~0\leqslant t\leqslant T,
\\  p^{\epss,\rho}(T)=&\lambda^{\varepsilon, \rho}\Phi_x(  X^{\varepsilon,\rho}(T))+\mu
^{\varepsilon, \rho}\phi_y(X^{\varepsilon,\rho}(T)),
\end{aligned}
\right.
\end{eqnarray}
\begin{eqnarray}\label{eq:4.5}
\left\{
\begin{aligned}
d p^{\epss}(t)=&-\bigg[A^*(t) p^{\epss}(t)+B(t)^* q^{\epss}(t)+ {\cal H}_{x}^{\epss} (t)\bigg]dt+ q^{\epss}(t)dW(t)+\displaystyle
\int_{{E}}r^{\epss}(t, e)\tilde{\mu}(de, dt),
~~~~0\leqslant t\leqslant T,
\\  p^{\epss}(T)=
&\lambda^{\varepsilon}\Phi_x(  X^{\varepsilon}(T))+\mu
^{\varepsilon}\phi_y(X^{\varepsilon}(T)),
\end{aligned}
\right.
\end{eqnarray}
and
\begin{eqnarray}\label{eq:4.6}
\begin{split}
   \left\{\begin{array}{ll}
d\bar p(t)=&-\bigg[A^*(t)\bar p(t)+B(t)^*\bar q(t)+\bar {\cal H}_{x} (t)\bigg]dt+\bar q(t)dW(t)+\displaystyle
\int_{{E}}\bar r(t, e)\tilde{\mu}(de, dt),
~~~~0\leqslant t\leqslant T,
\\ \bar p(T)=&\lambda\Phi_x( \bar X(T))
+\mu\phi_x(\bar X(T)),
  \end{array}
 \right.
 \end{split}
  \end{eqnarray}
respectively. In fact, the adjoint equations \eqref{eq:4.4}, \eqref{eq:4.5} and \eqref{eq:4.6} are three linear BSEEs satisfying Assumptions \ref{ass:3.1} and
\ref{ass:3.2}.
Hence by Lemma \ref{lem:1.3}, it is easy to check that these three adjoint equations have unique solutions, respectively.

\begin{lemma}\label{lem:4.1}
Under Assumptions \ref{ass:2.5}, the following convergence results hold
\begin{eqnarray}\label{eq:6.16}
\begin{split}
&\lim _{\rho\rightarrow 0} {\mathbb E} \bigg [ \sup_{0\leq t\leq T}\|p^{\varepsilon,\rho}(t)-{p}^{\varepsilon}(t)\|^2_H \bigg ]
+ {\mathbb E} \bigg [ \int_{0}^ T \|p^{\varepsilon, \rho}(t)-{p}^{\varepsilon}(t)\|^2_V d t
+ {\mathbb E} \bigg [ \int_{0}^ T \|q^{\varepsilon, \rho}(t)-{q}^{\varepsilon}(t)\|^2_H d t\bigg ]
\\ &\quad \quad+ {\mathbb E} \bigg [ \int_{0}^ T \|r^{\varepsilon, \rho}(t,\cdot)
-{r}^{\varepsilon}(t,\cdot)\|^2_{{M}_\mathscr{F}^{\nu,2}{([0,T]\times  E; H)}} d t\bigg]= 0 ,
\end{split}
\end{eqnarray}
and
\begin{eqnarray} \label{eq:6.13}
\begin{split}
&\lim _{\epss\rightarrow 0} {\mathbb E} \bigg [ \sup_{0\leq t\leq T}\|p^{\varepsilon}(t)
-{\bar p}(t)\|^2_H \bigg ]
+ {\mathbb E} \bigg [ \int_{0}^ T \|p^{\varepsilon}(t)-{\bar p}(t)\|^2_V d t
+ {\mathbb E} \bigg [ \int_{0}^ T \|q^{\varepsilon}(t)-{\bar q}(t)\|^2_H d t\bigg ]
\\ &\quad \quad+ {\mathbb E} \bigg [ \int_{0}^ T \|r^{\varepsilon}(t,\cdot)
-{\bar r}^{}(t,\cdot)\|^2_{{M}_\mathscr{F}^{\nu,2}{([0,T]\times  E; H)}} d t\bigg]=0
\end{split}
\end{eqnarray}
\end{lemma}

\begin{proof}
By the continuous dependence theorem of BSEE (i.e., Lemma \ref{lem:1.4}), we derive
\begin{eqnarray}
\begin{split}
&{\mathbb E} \bigg [ \sup_{0\leq t\leq T}\|p^{\varepsilon,\rho}(t)-{p}^{\varepsilon}(t)\|^2_H \bigg ]
+ {\mathbb E} \bigg [ \int_{0}^ T \|p^{\varepsilon, \rho}(t)-{p}^{\varepsilon}(t)\|^2_V d t
+ {\mathbb E} \bigg [ \int_{0}^ T \|q^{\varepsilon, \rho}(t)-{q}^{\varepsilon}(t)\|^2_H d t\bigg ]
\\ &\quad \quad+ {\mathbb E} \bigg [ \int_{0}^ T \|r^{\varepsilon, \rho}(t,\cdot)
-{r}^{\varepsilon}(t,\cdot)\|^2_{{M}_\mathscr{F}^{\nu,2}{([0,T]\times  E; H)}} d t\bigg] \nonumber \\
& \leq K \bigg \{ {\mathbb E} \bigg[ \int_0^T ||(b_x^{\varepsilon,\rho}(t)
-b_x^{\epss}(t))\cdot p^{\epss}(t)
+(g_x^{\varepsilon,\rho}(t)
-g_x^{\epss}(t))\cdot q^{\epss}(t)+
\int_E(\sigma_x^{\varepsilon,\rho}(t,e)
-\sigma_x^{\epss}(t,e))\cdot r^{\epss}(t,e)
\nu (de)
\\&\quad\quad+
 {\lambda}^{\varepsilon,\rho}l_x^{\varepsilon,\rho}(t)-{\lambda} ^{\varepsilon}l_x^{\varepsilon}(t)||^2_H d t \bigg ]+\mathbb E\bigg[||\lambda^{\varepsilon, \rho}\Phi_x(  X^{\varepsilon,\rho}(T))+\mu
^{\varepsilon, \rho}\phi_y(X^{\varepsilon}(T))-
\lambda^{\varepsilon}\Phi_x(  X^{\varepsilon}(T))-\mu
^{\varepsilon}\phi_y(X^{\varepsilon}(T))||_H.\bigg]
\bigg\}
 \end{split}
\end{eqnarray}
Then using \eqref{eq:5.9} and \eqref{eq:5.10} gives the desired result \eqref{eq:6.16}.
The proof of \eqref{eq:6.13} is similar and omitted here.
\end{proof}

%\begin{eqnarray}\label{eq:6.1}
%\left\{
%\begin{aligned}
%dk^{\varepsilon,\rho}(t) =& -\big[A^*(t)k^{\varepsilon,\rho}(t)+(b_y^{\varepsilon,\rho})^*(t)\cdot k^{\varepsilon,\rho}(t) +{\lambda}^{\varepsilon, \rho} l_y^{\varepsilon,\rho}(t)\big]dt
%-\big[(b_{z}^{\varepsilon,\rho})^*(t)\cdot k^{\varepsilon,\rho}(t)+{\lambda} ^{\varepsilon, \rho}l_{z}^{\varepsilon, \rho}(t)\big]dW(t), \\
%k(0) =& -{\lambda} ^{\varepsilon,\rho}\gamma_y(X^{\varepsilon,\rho}(T))-\mu^{\varepsilon,\rho}\phi_y(X^{\varepsilon,\rho}(T)),
%\end{aligned}
%\right.
%\end{eqnarray}
%\begin{eqnarray}\label{eq:6.3}
%\left\{
%\begin{aligned}
%dk^{\varepsilon}(t)=& -\big[A^*(t)k^{\varepsilon}(t)+(b_y^{\varepsilon})^*(t)\cdot k^{\varepsilon}(t) +{\lambda}^{\varepsilon} l_y^{\varepsilon}(t)\big]dt
%-\big[(b_{z}^{\varepsilon})^*(t)\cdot k^{\varepsilon}(t)+{\lambda}^{\varepsilon}l_{z}^{\varepsilon}(t)\big]dW(t), \\
%k(0)=& -{\lambda}^{\varepsilon}\gamma_y(y^{\varepsilon}(0))-\mu^{\varepsilon}\phi_y(y^{\varepsilon}(0)),
%\end{aligned}
%\right.
%\end{eqnarray}
%and
%\begin{eqnarray}\label{eq:6.5}
%\left\{
%\begin{aligned}
%d k(t)=& -\big[A^*(t) k(t)+\bar b_y^*(t)\cdot k(t) + {\lambda} \bar l_y(t)\big]dt-\big[\bar b_{z}^*(t)\cdot k(t)+\mu\bar l_{z}(t)\big]dW(t),\\
%k(0)=& -{\lambda}\gamma_y(\bar y(0))-\mu\phi_y(\bar y(0)),
%\end{aligned}
%\right.
%\end{eqnarray}

In the next lemma, we give a representation of the difference $J^{\varepsilon}(u^{\varepsilon,\rho}(\cdot))-J^\epss( u^\epss(\cdot))$
in terms of the Hamiltonian $H$, the adjoint process $(p^{\epss,\rho}(\cdot),
q^{\epss,\rho}(\cdot), r^{\epss,\rho}(\cdot,\cdot))$ and other relevant expressions associated with the admissible
pair $({u}^{\epss, \rho}(\cdot);{X}^{\epss,\rho}(\cdot))$.

\begin{lemma}\label{lem:4.3}
Under Assumptions \ref{ass:2.5}, it holds
\begin{eqnarray}\label{eq:6.15}
J^{\varepsilon}(u^{\varepsilon,\rho}(\cdot))-J^\epss(u^\epss(\cdot))
&=& {\mathbb E} \bigg [ \int_0^T \big\{ {\cal H}^{\epss, \rho} (t)
- {\cal H} (t,{X}^{\epss}(t),
{u}^{\epss}(t),{p}^{\epss,\rho}(t),q^{\epss, \rho}(t),r^{\epss, \rho}(t,\cdot),{\lambda}^{\epss, \rho}) \nonumber \\
&& -{\cal H}_x^{\epss, \rho}(t) \cdot(X^{\epss,\rho}(t)-X^{\epss}(t)) \big\} dt \bigg ] \nonumber \\
&& + \mu^{\varepsilon,\rho} {\mathbb E} \big [ \phi^{\epss,\rho}(X^{\epss,\rho}(T))
-\phi^{\epss}(X^{\epss}(T))
- \phi_x(X^{\epss,\rho}(T))
\cdot(X^{\epss,\rho}(T)-X^{\epss}(T)) \big ] \nonumber \\
&& + {\lambda}^{\varepsilon,\rho} {\mathbb E} \big [ \Phi^{\epss,\rho}(X^{\epss,\rho}(T))
-\Phi^{\epss}(X^{\epss}(T))
- \Phi_x(X^{\epss,\rho}(T))\cdot(X^{\epss, \rho}(T)-y^{\epss}(T)) \big ].
\end{eqnarray}
\end{lemma}

\begin{proof}
From the definition of the Hamiltonian $\cal H$ and $J^\epss( u(\cdot))$ (see \eqref{eq:3.17}), we deduce
\begin{eqnarray}\label{eq:4.10}
J^{\varepsilon}(u^{\varepsilon,\rho}(\cdot))-J^\epss( u^\epss(\cdot))
&=& {\lambda}^{\varepsilon,\rho}[{J(u^{\varepsilon,\rho}(\cdot))-J(u^\varepsilon(\cdot))}]
+\mu^{\varepsilon,\rho} {\mathbb E} \big [ \phi(X^{\varepsilon,\rho}(T)) - \phi(X^\varepsilon(T)) \big ] \nonumber \\
&=& {\mathbb E} \bigg [ \int_0^T \bigg\{ {\cal H}^{\epss, \rho} (t) - {\cal H} (t,{X}^{\epss}(t),
{u}^{\epss}(t),{p}^{\epss,\rho}(t),q^{\epss, \rho}(t),r^{\epss, \rho}(t,\cdot),{\lambda}^{\epss, \rho}) \nonumber \\
&& -  ( p^{\epss,\rho}(t), b^{\epss,\rho}(t)-b^{\epss}(t)))_H - (q^{\epss,\rho}(t), g^{\epss, \rho} (t)-g^{\epss} (t))_H \nonumber
\\&&-\int_{E}\bigg[(r^{\epss,\rho}(t,e)), \sigma^{\epss,\rho}(t,e)-\sigma^{\epss} (t,e))_H\nu(de)\bigg]\bigg\} dt \bigg ] \nonumber \\
&& + \mu^{\varepsilon,\rho} {\mathbb E} \big [ \phi(X^{\epss, \rho}(T))-\phi(X^{\epss}(T)) \big ]
+ {\lambda}^{\varepsilon,\rho} {\mathbb E} \big [ \Phi(X^{\epss, \rho}(T))-\Phi(X^{\epss}(T)) \big ] .
\end{eqnarray}
On the other hand,
\begin{eqnarray}
  \left\{
  \begin{aligned}
   d (X^{\epss,\rho} (t)-X^{\epss} (t))
   = & \ [ A (t)(X^{\epss,\rho} (t)-X^{\epss} (t))
   + (b ( t, X^{\epss,\rho} (t), u^{\epss,\rho}(t))-b ( t, X^{\epss} (t), u^{\epss}(t))) ] d t
\\&+ [ B(t)(X^{\epss,\rho} (t)-X^{\epss} (t))
   + (g( t, X^{\epss,\rho} (t), u^{\epss,\rho}(t))-g ( t, X^{\epss} (t), u^{\epss}(t))) ]d W(t)
 \\&+\int_E [\sigma( t,e, X^{\epss,\rho} (t), u^{\epss,\rho}(t))-\sigma ( t,e, X^{\epss} (t), u^{\epss}(t))) ]\tilde \mu(de,dt),  \\
X ^{\epss,\rho}(0)-X^\epss(0) = & \  0, \quad t \in [ 0, T ]
  \end{aligned}
  \right.
\end{eqnarray}
Then applying It\^{o} formula to $( p^{\epss,\rho}(t),   X^{\epss, \rho}(t)-X^{\epss}(t) )_H$ gives
\begin{eqnarray}\label{eq:4.12}
&& {\mathbb E} \bigg [ \int_0^T  \bigg\{ (p^{\epss,\rho}(t), b^{\epss, \rho} (t)-b^{\epss} (t))_H+(q^{\epss,\rho}(t), g^{\epss, \rho} (t)-g^{\epss} (t))_H
+\int_{E}(r^{\epss,\rho}(t,e), \sigma^{\epss, \rho} (t,e)-\sigma^{\epss} (t,e))_H\nu(de) \bigg\} dt \bigg ] \nonumber \\
&& = {\mathbb E} \bigg [ \int_0^T   {\cal H}_x^{\epss, \rho}(t) \cdot(X^{\epss,
\rho}(t)-X^{\epss}(t)) dt \bigg ] \nonumber +\mu^{\epss,\rho} {\mathbb E} \big [ \phi_x( X^{\epss}(T))\cdot(X^{\epss, \rho}(T)-X^{\epss}(T)) \big ]
\\&&~~~~+{\lambda}^{\epss,\rho} {\mathbb E} \big [ \Phi_x( X^\epss(T))\cdot(X^{\epss, \rho}(T)-X^{\epss}(T)) \big ].
\end{eqnarray}
Putting \eqref{eq:4.12} into \eqref{eq:4.10} leads to the desired representation \eqref{eq:6.15}.
\end{proof}

We have the following  basic Lemma.

\begin{lemma}\label{lem:3.2}
Under Assumptions \ref{ass:2.5}, it follows that
\begin{eqnarray}
\| X^{\varepsilon,\rho}(\cdot)
-{X}^\varepsilon(\cdot)\|_{{\cal M}^2_{\mathbb F}(0,T)}^2 = O (\rho^2),
\end{eqnarray}
and
\begin{eqnarray}
\| X^\varepsilon(\cdot)- {\bar X}(\cdot)\|_{{\cal M}^2_{\mathbb F}(0,T)}^2 = O (\varepsilon^2).
\end{eqnarray}
\end{lemma}

\begin{proof}
By the continuous dependence theorem of BSEE (Lemma \ref{lem:1.4}) and the uniform boundedness of the G\^{a}teaux derivative $b_u$, we have
\begin{eqnarray*}
&&\| X^{\varepsilon,\rho}(\cdot)
-{X}^\varepsilon(\cdot)\|_{{\cal M}^2_{\mathbb F}(0,T)}^2
\\&\leq& K {\mathbb E} \bigg [ \int_0^T
\bigg\{\|b (t, {X}^\varepsilon(t),  u^{\varepsilon,\rho}(t)) - b^\varepsilon (t)\big\|^2_Hdt
+\| g(t, {y}^\varepsilon(t), z^\varepsilon(t), u^{\varepsilon,\rho}(t)) - g^\varepsilon (t)\big\|^2_H
\\&&\quad\quad+\int_E\bigg[\|\sigma (t, e,{X}^\varepsilon(t), u^{\varepsilon,\rho}(t)) - \sigma^\varepsilon (t)\big\|^2_H\bigg]\nu(de)\bigg\}dt\bigg] \\
&\leq& K {\mathbb E} \bigg [ \int_0^T \|u^{\varepsilon,\rho}(t)-{u^\epss}(t)\|^2_Udt\bigg] \\
&=& K \rho^2 {\mathbb E} \bigg[ \int_0^T \|v(t)-{u^\varepsilon}(t)\|^2_Udt\bigg] \\
&\leq& K \rho^2  \\
&=& O(\rho^2).
\end{eqnarray*}
Here $K$ is a generic positive constant and might change from line to line.

In the same vein, we deduce
\begin{eqnarray*}
\| X^\varepsilon(\cdot)-  {\bar X}(\cdot)\|_{{\cal M}^2_{\mathbb F}(0,T)}^2
&\leq& K {\mathbb E} \bigg[ \int_0^T \|u^{\varepsilon}(t)-{\bar u}(t)\|^2_Udt\bigg] \\
&=& K d^2(u^{\varepsilon}(t),{\bar u}(t)).\\
&\leq& K \varepsilon^2\\
&=& O(\epss).
\end{eqnarray*}
The proof is complete.
\end{proof}

Now we state the variational formula for the cost functional $J^\epss(\cdot)$.

\begin{theorem}\label{them:3.1}
Under Assumptions \ref{ass:2.5}, it follows that
for any admissible control $v(\cdot),$ the cost functional $J(u(\cdot))$ is
G\^{a}teaux differentiable at $ u^\epss(\cdot)$  in the direction $v(\cdot)- u^\epss(\cdot)$
and the corresponding G\^{a}teaux derivative $J'$ is given by
\begin{eqnarray}\label{eq:4.16}
\frac{d}{d\rho}J^\epss( u^\epss(\cdot)+\rho(v(\cdot)- u^\epss(\cdot)))|_{\rho=0}
&=& \lim_{\rho\rightarrow 0} \frac{J^\epss(u^\epss(\cdot)+\rho(v(\cdot)- u^\epss(\cdot)))-J^\epss(u^\epss(\cdot))}{\rho} \nonumber \\
&=& {\mathbb E} \bigg [ \int_0^T( {\cal H}_u^\epss(t), v(t)-{u}^\epss(t))_Udt \bigg ]\nonumber
\\
&\geq& -C\varepsilon ^{\frac{1}{2}}.
\end{eqnarray}
Here $\rho >0$ is a sufficiently small positive constant. %and
\end{theorem}

\begin{proof}
By \eqref{eq:6.15}, we have
\begin{eqnarray}\label{eq:6.23}
&& J^\epss( u^\epss(\cdot)+\rho(v(\cdot)- u^\epss(\cdot)))-J^\epss(u^\epss(\cdot)) = I + II ,
\end{eqnarray}
where
\begin{eqnarray*}
I &\triangleq& {\mathbb E} \bigg [ \int_0^T \big\{ {\cal H}^{\epss, \rho} (t)
- {\cal H} (t,{X}^{\epss}(t),
{u}^{\epss}(t),{p}^{\epss,\rho}(t),q^{\epss, \rho}(t),r^{\epss, \rho}(t,\cdot),{\lambda}^{\epss, \rho}) \nonumber \\
&& -{\cal H}_x^{\epss, \rho}(t) \cdot(X^{\epss,\rho}(t)-X^{\epss}(t))
 -{\cal H}_u^{\epss, \rho}(t) \cdot(u^{\epss,\rho}(t)-u^{\epss}(t))\big\} dt \bigg ] \nonumber \\
&& + \mu^{\varepsilon,\rho} {\mathbb E} \big [ \phi^{\epss,\rho}(X^{\epss,\rho}(T))
-\phi^{\epss}(X^{\epss}(T))
- \phi_x(X^{\epss,\rho}(T))
\cdot(X^{\epss,\rho}(T)-X^{\epss}(T)) \big ] \nonumber \\
&& + {\lambda}^{\varepsilon,\rho} {\mathbb E} \big [ \Phi^{\epss,\rho}(X^{\epss,\rho}(T))
-\Phi^{\epss}(X^{\epss}(T))
- \Phi_x(X^{\epss,\rho}(T))\cdot(X^{\epss, \rho}(T)-y^{\epss}(T)) \big ].
\end{eqnarray*}
and
\begin{eqnarray*}
II &\triangleq& {\mathbb E} \bigg [ \int_0^T {\cal H}_u^{\epss, \rho}(t) \cdot(u^{\epss, \rho}(t)-u^{\epss}(t)) d t \bigg ]
\end{eqnarray*}
Recalling Lemma \ref{lem:3.2} and Assumption \ref{ass:2.5} and using the Taylor Expansion for $H$ and the
dominated convergence theorem, we obtain
\begin{eqnarray}\label{eq:4.14}
I = o(\rho).
\end{eqnarray}
On the other hand, similarly, using Lemma \ref{lem:4.1}, Lemma \ref{lem:3.2} and Assumption \ref{ass:2.5} and using the Taylor Expansion for $H$ and the
dominated convergence theorem, we deduce
\begin{eqnarray}\label{eq:4.19}
II =\rho{\mathbb E} \bigg [ \int_0^T( {\cal H}_u^\epss(t), v(t)-{u}^\epss(t))_Udt \bigg ]+o(\rho)
\end{eqnarray}
Hence, putting \eqref{eq:4.14} and
 \eqref{eq:4.19}into \eqref{eq:6.23} and combing
\eqref{eq:3.9}, by  the
dominated convergence theorem we conclude that
\begin{eqnarray}
\frac{d}{d\rho}J^\epss( u^\epss(\cdot)+\rho(v(\cdot)- u^\epss(\cdot)))|_{\rho=0}
&=& \lim_{\rho\rightarrow 0} \frac{J^\epss(u^\epss(\cdot)+\rho(v(\cdot)- u^\epss(\cdot)))-J^\epss(u^\epss(\cdot))}{\rho} \nonumber \\
&=& {\mathbb E} \bigg [ \int_0^T( {\cal H}^\epss_u(t), v(t)-{u}^\epss(t))_U d t \bigg ]
\geq - C \varepsilon^{\frac{1}{2}} .
\end{eqnarray}
\end{proof}

Now we are ready to give the necessary condition of optimality for the existence of the optimal control of Problem \ref{pro:2.1}.

\begin{theorem}
Let Assumptions \ref{ass:2.5} be satisfied. Let $(\bar{u}(\cdot); \bar{X}(\cdot))$ be an optimal pair of Problem \ref{pro:2.1}.
Then there exist a $({\lambda},\mu)$ satisfying $|{\lambda}|^2+|\mu|^2=1$ such that
\begin{eqnarray}\label{eq:4.20}
( {\cal H}_u(t,{\bar X}(t),{\bar u}(t),{\bar p}(t),\bar q(t),\bar r(t,\cdot), {\lambda}), u-{\bar u}(t))_U\geq 0, \quad \forall u\in U_{ad}, \quad \mbox{a.e.} \ \mbox{a.s.}.
\end{eqnarray}
Here $\{ {\bar p}(t), \bar q(t),
 \bar r(t,\cdot)),0 \leq t \leq T \}$ be  the solution of the corresponding adjoint equation \eqref{eq:6.10} associated
with $(\bar{u}(\cdot); \bar{X}(\cdot))$.
\end{theorem}

\begin{proof}
From \eqref{eq:5.14}, there exists a pair $({\lambda},\mu)$ satisfying $|{\lambda}|^2+|\mu|^2=1$. Note that
\begin{eqnarray}\label{eq:4.22}
\lim_{\epss \rightarrow 0} d\left( u^{\varepsilon }\left( \cdot \right) ,{\bar u}\left( \cdot \right) \right) =0
\end{eqnarray}
From \ref{lem:4.1}, Lemma \ref{lem:3.2} and Assumption \ref{ass:2.5}  and \eqref{eq:5.13}, sending $\varepsilon$ to $0$ on the both sides of  \eqref{eq:4.16} and using
the dominated convergence theorem, we conclude that
\begin{eqnarray}
{\mathbb E} \bigg [ \int_0^T( {\cal H}_u(t,{\bar X}(t),{\bar u}(t),{\bar p}(t),\bar q(t),\bar r(t,\cdot), {\lambda}), v(t)-{\bar u}(t))_Udt \bigg ] \geq 0, \quad \forall v (\cdot) \in \cal A,
\end{eqnarray}
which implies that
\eqref{eq:4.20} holds.
This completes the proof.
\end{proof}

\bibliographystyle{amsplain}

\vspace{1mm}

\section*{Appendix}

\setcounter{equation}{0}
\renewcommand{\theequation}{A.\arabic{equation}}
\renewcommand{\thedefinition}{A.\arabic{definition}}
\renewcommand{\thetheorem}{A.\arabic{theorem}}
\renewcommand{\thelemma}{A.\arabic{lemma}}
\renewcommand{\theassumption}{A.\arabic{assumption}}

In this appendix, we introduce some preliminary results of SEEs and BSEEs, including
existence, uniqueness and continuous dependence theorems.

Consider a SEE in the Gelfand triple $(V, H, V^*)$:
\begin{eqnarray} \label{eq:3.1}
  \left\{
  \begin{aligned}
   d X (t) = & \ [ A (t) X (t) + b ( t, X (t)) ] d t
+ [B(t)X(t)+g( t, { X (t)}) ]d W(t)
 \\&\quad +\int_E \sigma (t,e, X(t-))\tilde \mu(de,dt),  \\
X (0) = & \  x \in H , \quad t \in [ 0, T ],
  \end{aligned}
  \right.
\end{eqnarray}
where $A,B,b,g$ and  $\sigma $ are given random mappings
which satisfy the following
standard  assumptions.

\begin{assumption} \label{ass:3.1}
  The operator processes $A:[0,T]\times \Omega \longrightarrow {\mathscr L} (V, V^*)$ and $B
  : [0,T]\times \Omega \longrightarrow {\mathscr L} (V, H)$
  are weakly predictable; i.e.,
  $ \langle A(\cdot)x, y \rangle$ and $(B(\cdot)x, y)_H$
  are both predictable process for every $x, y\in V, $
  and satisfy the coercive condition, i.e., there exist
  some constants  $ C, \alpha>0$ and $\lambda$ such that for any $x\in V$ and  each $(t,\omega)\in [0,T]\times \Omega,$
    \begin{eqnarray}
    \begin{split}
     - \langle A(t)x, x \rangle +\lambda ||x||_H^2 \geq \alpha
      ||x||_V^2+||Bx||_H^2{\color{blue},}
    \end{split}
  \end{eqnarray}
  and  \begin{eqnarray}
\sup_{( t, \omega ) \in [0, T] \times \Omega} \| A ( t,\omega ) \|_{{\mathscr L} ( V, V^* )}
 +\sup_{( t, \omega ) \in [0, T] \times \Omega} \| B ( t,\omega ) \|_{{\mathscr L} ( V, H )} \leq C \ .
\end{eqnarray}
\end{assumption}

\begin{assumption} \label{ass:3.2}
   The mappings $b:[0,T]\times  \Omega\times H  \longrightarrow H$ and $g:[0,T]\times\Omega\times H  \longrightarrow H$  are both $\mathscr P\times
   \mathscr B(H)/\mathscr B(H) $-measurable
   such that $b(\cdot,0), g(\cdot,0)\in M_{\mathscr{F}}^2(0,T;H)
   $; the mapping $\sigma:[0,T]\times  \Omega \times
   E\times H  \longrightarrow H$ is $\mathscr P\times\mathscr B(E)  \times
   \mathscr B(H)/\mathscr B(H) $-measurable
   such that $\sigma(\cdot,\cdot, 0)\in {M}_\mathscr{F}^{\nu,2}{([0,T]\times  E; H)}$.
   And there exists a constant $C$  such that
   for all $x, \bar x\in V$ and a.s.$(t,\omega)\in [0,T]\times \Omega,$
   \begin{eqnarray}
     \begin{split}
       ||b(t,x)-b(t,x)||_H+ ||g(t,x)-g(t,x)||_H
       +||\sigma(t,\cdot, x)-\sigma(t,\cdot,x)||
       _{M^{\nu,2}( E; H)} \leq  C||x-\bar x||_H.
     \end{split}
   \end{eqnarray}

\end{assumption}

\begin{definition}
\label{defn:c1}
A $V$-valued, $\{{\mathscr F}_t\}_{0\leq t\leq T}$-adapted process $X(\cdot)$ is said to be a solution to the
SEE \eqref{eq:3.1}, if $X (\cdot) \in { M}_{
\mathscr F}^2 ( 0, T; V )$ such that for every $\phi \in V$
and a.e. $( t, \omega ) \in [0, T ] \times \Omega$, it holds that
\begin{eqnarray}
\left\{
\begin{aligned}
( X (t), \phi )_H =& \ ( x, \phi )_H + \int_0^t \left < A (s) X (s), \phi \right > d s
+\int_0^t ( b ( s, X (s){\color{blue})}, \phi )_H d s \\
& + \int_0^t ( B(s)X(s)+ g ( s, X (s) ), \phi )_H d W (s)
\\& + \int_0^t \int_{E} (\sigma ( s,e, X (s-) ), \phi )_H d \tilde \mu (de,ds), \quad t \in [ 0, T ] , \\
X(0) =& \ x\in H  ,
\end{aligned}
\right.
\end{eqnarray}
or alternatively, $X (\cdot)$ satisfies the following It\^o's equation in $V^*$:
\begin{eqnarray}
\left\{
\begin{aligned}
X (t)=& \  x+ \int_0^t  A (s) X (s)d s
+\int_0^t  b ( s, X (s))d s + \int_0^t
[B(s)X(s)+ g ( s, X (s) )] d W (s)
\\& + \int_0^t \int_{E} \sigma ( s,e, X (s-) ) d \tilde \mu (de,ds), \quad t \in [ 0, T ] , \\
X(t) =& \ x\in H .
\end{aligned}
\right.
\end{eqnarray}
\end{definition}

Now we state our main result.

\begin{lemma} \label{thm:3.1}
  Let Assumptions \ref{ass:3.1}-\ref{ass:3.2} be
  satisfied by any given coefficients
  $(A,B,b,g,\sigma)$ of the SEE \eqref{eq:3.1}. Then for any initial
    value $X(0)=x,$ the
    SEE \eqref{eq:3.1} has a unique
  solution $X(\cdot)\in M_{\mathscr{F}}^2(0,T;V) \bigcap  S_{\mathscr{F}}^2(0,T;H).$
\end{lemma}
To prove this  theorem, we first show the following
result on  the continuous dependence of  the solution to  the  SEE \eqref{eq:3.1}.

 \begin{lemma} \label{thm:3.2}
  Let  $ X(\cdot)$  be a solution to
  the  SEE   \eqref{eq:3.1}
   with the initial value $X(0)=x$ and
    the coefficients $(A,B, b,g,\sigma)$
   which satisfy Assumptions \ref{ass:3.1}-\ref{ass:3.2}. Then
  the following estimate holds:
\begin{eqnarray}\label{eq:3.4}
\begin{split}
&{\mathbb E} \bigg [ \sup_{0 \leq t \leq T} \| X (t) \|_H^2 \bigg
]
+ {\mathbb E} \bigg [ \int_0^T \| X (t) \|_V^2 d t \bigg ] \\
& \leq K \bigg \{ ||x||_H^2 + {\mathbb E}
\bigg [ \int_0^T \| b ( t, 0) \|_H^2 d t \bigg ] + {\mathbb E}
\bigg [ \int_0^T \| g ( t, 0) \|_H^2 d t \bigg ]
+ {\mathbb E} \bigg [ \int_{0}^T\int_E \| \sigma (t,e,0) \|^2_H  \nu(de)d t \bigg ] \bigg \}.
\end{split}
\end{eqnarray}
Furthermore, suppose that   $ \bar X(\cdot)$  is a solution to
  the  SEE   \eqref{eq:3.1}
   with the initial value $\bar X(0)=\bar x \in H$ and the coefficients $(A,B,  \bar b, \bar g,\bar\sigma)$
   satisfying Assumptions \ref{ass:3.1}-\ref{ass:3.2},
   then we have

   \begin{eqnarray}\label{eq:3.5}
&& {\mathbb E} \bigg [ \sup_{0 \leq t \leq T} \| X (t) - {\bar X} (t) \|_H^2 \bigg ]
+ {\mathbb E} \bigg [ \int_0^T \| X (t) - {\bar X} (t) \|_V^2 d t \bigg ] \nonumber \\
&& \leq K \bigg \{  \|x-\bar x\|^2_H
+ {\mathbb E} \bigg [ \int_0^T \| b ( t, {\bar X} (t) )
- {\bar b} ( t, {\bar X} (t) ) \|_H^2 d t \bigg ] \\
&&+ {\mathbb E} \bigg [ \int_0^T \| g ( t, {\bar X} (t) ) - {\bar g}( t, {\bar X} (t) ) \|_H^2 d t \bigg ]
+ {\mathbb E} \bigg [ \int_0^T \int_{E}\| \sigma ( t, e, {\bar X} (t)) - {\bar \sigma}( t, e, {\bar X} (t) ) \|_H^2 \nu(de)d t \bigg ]   \bigg \} .\nonumber
\end{eqnarray}
 \end{lemma}

Next we consider a BSEE in the Gelfand triple $(V, H, V^*)$:
\begin{eqnarray}\label{eq:2.7}
\left\{
\begin{aligned}
dY(t)=& \ [ A^*(t)Y(t)+B^*(t)Z(t)
+f(t, Y(t), Z(t), R(t,\cdot)) ] dt +Z(t)dW(t)+\int_{E}R(t,e)\tilde\mu(dt,de),\\
Y(T)=& \ \xi,
\end{aligned}
\right.
\end{eqnarray}
where  $(A^*,B^*,f, \xi)$  are given random mappings. Here $A^*$ and $B^*$
are the adjoint operators of $A$ and
$B$, respectively. Furthermore, we assume that the coefficients $(A^*, B^*, f, \xi)$
satisfy the following conditions:

\begin{assumption} \label{ass:3.1}
  The operator processes $A^*:[0,T]\times \Omega \longrightarrow {\mathscr L} (V, V^*)$ and $B^*
  : [0,T]\times \Omega \longrightarrow {\mathscr L} (V, H)$
  are weakly predictable; i.e.,
  $ \langle A^*(\cdot)x, y \rangle$ and $(B^*(\cdot)x, y)_H$
  are both predictable process for every $x, y\in V, $
  and satisfy the coercive condition, i.e.,  there exist
  some constants  $ C, \alpha>0$ and $\lambda$ such that for any $x\in V$ and  each $(t,\omega)\in [0,T]\times \Omega,$
    \begin{eqnarray}
    \begin{split}
     - \langle A^*(t)x, x \rangle +\lambda ||x||_H \geq \alpha
      ||x||_V+||B^*x||_H{\color{blue},}
    \end{split}
  \end{eqnarray}
  and  \begin{eqnarray}
\sup_{( t, \omega ) \in [0, T] \times \Omega} \| A^* ( t,\omega ) \|_{{\mathscr L} ( V, V^* )}
 +\sup_{( t, \omega ) \in [0, T] \times \Omega} \| B^* ( t,\omega ) \|_{{\mathscr L} ( V, H )} \leq C \ .
\end{eqnarray}
\end{assumption}

\begin{assumption} \label{ass:3.2}
   The mapping $\xi: \Omega \rightarrow H$ is ${\cal F}_T$-measurable such that
   $\xi\in L^2(\Omega,{\mathscr{F}_T},
   \mathbb P;H).$
    The mappings $f:[0,T]\times  \Omega\times H \times H\times
    M^{\nu,2}( E; H) \longrightarrow $   are both $\mathscr P\times
   \mathscr B(H)\times \mathscr B(H)
   \times
   \mathscr B(M^{\nu,2}( E; H))/\mathscr B(H) $-measurable
   such that $f(\cdot,0,0,0)\in M_{\mathscr{F}}^2(0,T;H)
   $.
   And there exists a constant $C$  such that
   for all $$(t,y,z,r,\bar{y},\bar{z},
   \bar r)
\in [0, T]\times H\times H\times M^{\nu,2}( E; H)\times H\times H \times
M^{\nu,2}( E; H)$$ and a.s.$(t,\omega)\in [0,T]\times \Omega,$
   \begin{eqnarray}
     \begin{split}
       ||f(t,y,z,r)-b(t,y,z,r)||_H \leq  C\bigg\{||y-\bar y||_H+||z-\bar z||_H+
       ||r-\bar r||_{M^{\nu,2}( E; H)}\bigg\}.
     \end{split}
   \end{eqnarray}
\end{assumption}

If the coefficients $(A^*, B^*,f,\xi)$ satisfy Assumptions \ref{ass:3.1} and
\ref{ass:3.2}, they are said to be a generator of BSEE \eqref{eq:2.7}.

\begin{definition}\label{defn:c1}
A  $(V \times H \times M^{\nu,2}( E; H) )$-valued, ${\mathbb F}$-adapted process $( Y (\cdot), Z (\cdot), R (\cdot, \cdot) )$
is called a solution to the BSEE \eqref{eq:2.7}, if $Y (\cdot) \in { M}_{\mathscr F}^2 ( 0, T; V )$,
$Z (\cdot) \in { M}_{\mathscr F}^2 ( 0, T; H )$ and $R (\cdot, \cdot) \in { M}_{\mathscr F}^{\nu, 2}(0,T; H)$ such that for every $\phi \in V$ and
a.e. $( t, \omega ) \in [ 0, T ] \times \Omega$, it holds that
\begin{eqnarray}\label{eq:c5}
( Y (t),  \phi )_H  &=& (\xi, \phi)_H
- \int_t^T \Big\langle  A^* (s) Y (s) +B^*(s)Z(s)+f ( s, Y (s), Z (s), Y ( s  ), R(s, \cdot) ), \phi \Big\rangle d t \nonumber \\
&& - \int_t^T ( Z (s), \phi )_H d W (s) -\int_t^T\int_E (R(s,e), \phi)_H\tilde \mu(ds,de), \quad t \in [0, T] ,
\end{eqnarray}
or alternatively, $( Y (\cdot), Z (\cdot), R (\cdot, \cdot) )$ satisfies the following It\^{o}'s equation in $V^*$:
\begin{eqnarray}
Y(t)&=& \xi -\int_t^T \big[A^* (s) Y (s) d s+ B^* (s) Z (s)+f ( t, {Y} (s), {Z} (s),  R(s,\cdot) ) \big] d s \nonumber \\
&&- \int_t^T Z (s) d W (s)-\int_t^T \int_E R (s,e) d \tilde \mu(ds,de) , \quad t \in [0, T] .
\end{eqnarray}
\end{definition}

\begin{lemma}[{\bf Existence and Uniqueness of BSEE \cite{Me}}]\label{lem:1.3}
For any generator $(A^*,B^*, f,\xi)$, BSEE \eqref{eq:2.7} has a unique solution $(Y(\cdot), Z(\cdot), R(\cdot,\cdot)).$ Moreover, $Y(\cdot)\in S_{\mathbb{ F}}^2(0,T;H)$.
\end{lemma}

\begin{lemma}[{\bf  Continuous Dependence Theorem of BSEE}]\label{lem:1.4}
Let $(A^*, B^*, f,\xi)$ and $(A^*,
B^*,  \bar{f},\bar{\xi})$ be two generators of BSEE \eqref{eq:2.7}. Suppose that $(Y(\cdot),Z(\cdot), R(\cdot,\cdot))$ and $(\bar{Y}(\cdot),\bar{Z}(\cdot),
\bar R(\cdot,\cdot))$ are the
solutions of BSEE \eqref{eq:2.7} corresponding to $(A^*,B^*, f,\xi)$ and $(A^*,B^*, \bar f,\bar{\xi})$, respectively. Then
\begin{eqnarray}\label{eq:2.14}
&& {\mathbb E} \bigg [ \sup_{t \in [0, T]}\|Y(t) -\bar{Y}(t)\|_H^2 \bigg ] + {\mathbb E} \bigg [ \int_{0}^T\|Y(t)-\bar{Y}(t)\|_V^2dt \bigg ]
+ {\mathbb E} \bigg [ \int_0^T\|Z(t)-\bar{Z}(t)\|^2_Hdt \bigg ]
\nonumber
\\&&~~~+{\mathbb E} \bigg [ \int_0^T\int_E\|R(t,e)-\bar{R}(t,e)\|^2_H
\nu(de)dt \bigg ]  \nonumber \\
&& \leq K \bigg \{ {\mathbb E} [ \|\xi-\bar{\xi}\|_H^2 ] + {\mathbb E} \bigg [ \int_0^T\|f(t,\bar{Y}{(t),\bar{Z}(t),
\bar R(t,\cdot)})-\bar{f}(t, \bar{Y}{(t),\bar{Z}(t)},
\bar R(t,\cdot))\|^2_Hdt \bigg ] \bigg \} ,
\end{eqnarray}
where $K$ is a positive constant depending only on $T$ and the constants $C,\alpha, {\lambda}$ in Assumption \ref{ass:3.1}.

In particular, if $(A^*, B^*,\bar{f},\bar{\xi})=(A^*, B^*, 0,0)$, the following a priori estimate holds
\begin{eqnarray}\label{eq:2.15}
&&{\mathbb E} \bigg [ \sup_{t \in [0, T]}\|Y(t)\|^2_H \bigg ] + {\mathbb E} \bigg [ \int_{0}^T\|Y(t)\|_V^2 d t \bigg ] + {\mathbb E} \bigg [ \int_0^T\|Z(t)\|^2_Hdt
 + {\mathbb E} \bigg [ \int_0^T\int_{E}\|R(t,e)\|^2_H\nu(de)dt\bigg ]
\\&&\leq K \bigg\{ {\mathbb E} [ \|\xi\|^2_H ] + {\mathbb E} \bigg [ \int_0^T\|f(t,0,0,0) \|^2_Hdt \bigg ] \bigg \}.
\end{eqnarray}
\end{lemma}

\end{document}